\documentclass[12pt]{amsart}
\usepackage{amssymb}
\usepackage{mathtools}
\usepackage[english]{babel}
\usepackage{epsfig}
\numberwithin{equation}{section}

\def\Re{{\sf Re}\,}
\def\Im{{\sf Im}\,}

\def\1#1{\overline{#1}}
\def\2#1{\widetilde{#1}}
\def\3#1{\widehat{#1}}
\def\4#1{\mathbb{#1}}
\def\5#1{\frak{#1}}
\def\6#1{{\mathcal{#1}}}

\newcommand{\R}{\mathbb R}
\newcommand{\Ha}{\mathbb H}

\newcommand{\C}{\mathbb C}

\newcommand{\D}{\mathbb D}
\newcommand{\oD}{\overline{\mathbb D}}

\newcommand{\N}{\mathbb N}

\def\Re{{\sf Re}\,}
\def\Im{{\sf Im}\,}

\newcommand{\mcite}[1]{\csname b@#1\endcsname}

\theoremstyle{theorem}

\newtheorem*{theorem*}{Theorem}

\def\Re{{\sf Re}\,}
\def\Im{{\sf Im}\,}

\newtheorem{theorem}{Theorem}[section]
\newtheorem{lemma}[theorem]{Lemma}
\newtheorem{proposition}[theorem]{Proposition}
\newtheorem{corollary}[theorem]{Corollary}

\theoremstyle{definition}
\newtheorem{definition}[theorem]{Definition}

\theoremstyle{remark}

\numberwithin{equation}{section}

\title[Orthogonal convergence]{A characterization of orthogonal convergence in simply connected domains}

\author[F. Bracci]{Filippo Bracci$^\ast$}
\address{F. Bracci: Dipartimento di Matematica, Universit\`a di Roma ``Tor Vergata", Via della Ricerca
Scientifica 1, 00133, Roma, Italia.} \email{fbracci@mat.uniroma2.it}

\author[M. D. Contreras]{Manuel D. Contreras$^\dag$}

\author[S. D\'{\i}az-Madrigal]{Santiago D\'{\i}az-Madrigal$^\dag$}
\address{M. D. Contreras, S. D\'{\i}az-Madrigal: Camino de los Descubrimientos, s/n\\
Departamento de Matem\'{a}tica Aplicada~II and IMUS\\ Universidad de Sevilla\\ Sevilla,
41092\\ Spain.}\email{contreras@us.es} \email{madrigal@us.es}

\author[H. Gaussier]{Herv\'e Gaussier}
\address{H. Gaussier: Univ. Grenoble Alpes, CNRS, IF, F-38000 Grenoble, France.}\email{herve.gaussier@univ-grenoble-alpes.fr}


\thanks{$^\dag$ Partially supported by the \textit{Ministerio
de Econom\'{\i}a y Competitividad} and the European Union (FEDER) MTM2015-63699-P and  by \textit{La Consejer\'{\i}a de Educaci\'{o}n y Ciencia de la Junta de Andaluc\'{\i}a}.}

\thanks{$^\ast$ Partially supported by the MIUR Excellence Department Project awarded to the  
Department of Mathematics, University of Rome Tor Vergata, CUP E83C18000100006}

\long\def\REM#1{\relax}

\begin{document}
\maketitle

\begin{abstract} Let $\mathbb D$ be the unit disc in $\mathbb C$ and let $f:\mathbb D \to \mathbb C$ be a Riemann map, $\Delta=f(\mathbb D)$. We give a necessary and sufficient condition in terms of hyperbolic distance and horocycles which assures that a  compactly divergent sequence $\{z_n\}\subset \Delta$ has the property that $\{f^{-1}(z_n)\}$ converges orthogonally to a point of $\partial \mathbb D$. We also give some applications of this to the slope problem for continuous semigroups of holomorphic self-maps of $\mathbb D$.
\end{abstract}

\tableofcontents

\section{Introduction}

A sequence $\{\zeta_n\}\subset \D:=\{\zeta\in \C: |\zeta|<1\}$ is said to converge {\sl orthogonally} to a point $\sigma\in \partial \D$ provided $\{\zeta_n\}$ converges to $\sigma$ and $\lim_{n\to \infty}\arg(1-\overline{\sigma}\zeta_n)=0$.

Let $\Delta\subsetneq \C$ be a simply connected domain, and $\{z_n\}\subset \Delta$ a sequence with no accumulation points in $\Delta$. Let $f:\D \to \Delta$ be a Riemann map. The aim of this note is to give an answer to the following question:

\medbreak
{\sl What are (useful) geometric conditions on $\Delta$ which detect whether $\{f^{-1}(z_n)\}$ converges orthogonally to a point $\sigma\in \partial \D$?}
\medbreak

The previous question, aside being interesting by itself, is particularly important in studying dynamics of  continuous semigroups of holomorphic self-maps of the unit disc (or more generally of holomorphic self-maps of the unit disc). Indeed, every semigroup of holomorphic self-maps of the unit disc has an essentially unique holomorphic model where the dynamical properties of the original semigroup translates into geometrical properties of the base domain of the model.

A similar question for non-tangential convergence has been settled in the recent paper \cite{BCDG}, using hyperbolic geometry.

A remarkable result in the direction of the previous question has been proved using harmonic measure theory by D. Betsakos \cite[Theorem 2]{Bet}:

\begin{theorem*}[Betsakos]
Let $\Delta\subsetneq \C$ be a simply connected domain starlike at infinity (namely $\Delta+it\subset \Delta$ for all $t\geq 0$) and $f:\D\to \Delta$ a Riemann map. If $\partial \Delta\subset\{z\in \C: a<\Re z<b, \Im z<c\}$ for some $-\infty<a<b<+\infty$ and $c\in \R$, then there exists $\tau\in \partial \D$ such that $f^{-1}(w+it)$ converges orthogonally to $\tau$ for all $w\in \Delta$.
\end{theorem*}

In this paper we give a complete answer to the above question in terms of hyperbolic geometry and ``horocycles'' in the spirit of \cite{BG}, explaining also more geometrically Betsakos's theorem.

For the sake of clearness, we give here a definition of horocycles, referring the reader to Section \ref{horocycles} for details. Let $\underline{y}$ be a prime end of the simply connected domain $U\subsetneq \C$. Let $f:\D\to U$ be a Riemann map such that $1$ corresponds to $\underline{y}$ under $f$. Let $z_0:=f(0)$. For $R>0$, the {\sl horocycle} $E^\Delta_{z_0}(\underline{y},R)$ centered at $\underline{y}$, with base point $z_0$ and radius $R$, is given by $f(E(1,R))$, where $E(1,R):=\{z\in \D: |1-z|^2<R(1-|z|^2)\}$ is a classical horocycle in the unit disc.

If $U \subsetneq \mathbb C$ is a simply connected domain, we denote by $\partial_CU$ the set of prime ends of $U$ and by $\widehat{U}:=U \cup \partial_C U$, endowed with the Carath\'eodory prime ends topology, or {\sl Carath\'eodory topology of $U$} for short
(see \cite{CL} for precise definitions).

\vspace{1mm}
The main result of this paper is the following (see Section \ref{geodesics} for definitions and properties of hyperbolic distance and geodesics):

\begin{theorem}\label{Thm:main}
Let $\Delta\subsetneq \C$ be a simply connected domain, $f:\D \to \Delta$ a Riemann map. Let $\{z_n\}\subset \Delta$ be a sequence with no accumulation points in $\Delta$. Then there exists $\sigma\in \partial \D$ such that $\{f^{-1}(z_n)\}$ converges orthogonally to $\sigma$ if and only if  there exist a simply connected domain $U\subsetneq \C$, $z_0\in U$, $\underline{y}\in \partial_CU$ and $R>0$ such that
\begin{enumerate}
\item $E_{z_0}^U(\underline{y}, R)\subset\Delta\subseteq U$,
\item $\lim_{n\to \infty}k_U(z_n, \gamma([0,+\infty)))=0$, where $\gamma:[0,+\infty)\to U$ is any geodesic for the hyperbolic distance in $U$ such that $\lim_{t\to+\infty}\gamma(t)=\underline{y}$ in the Carath\'eodory topology of $U$.
\end{enumerate}
In particular,  $\gamma(t)$ is eventually contained in $\Delta$ and $f^{-1}(\gamma(t))$ converges orthogonally to~$\sigma$.
\end{theorem}

Betsakos' Theorem can be seen then as a consequence of Theorem~\ref{Thm:main}: the domain starlike at infinity $\Delta$ is contained in a Koebe domain $K:=\C\setminus \{z: \Re z=s_0, \Im z\leq s_1\}$ for some $s_0, s_1\in \R$, and the condition that $\partial \Delta$ is contained in a vertical semistrip implies that $\Delta$ contains a horocycle of $K$ centered at the prime end corresponding to ``infinity''. The line $t\mapsto s_0+it$ is a geodesic in $K$ (for $t>>1$) which converges to ``infinity'', and hence one gets the orthogonal convergence of $f^{-1}(s_0+it)$ as $t\to+\infty$ (and of $f^{-1}(w+it)$ for all $w\in \Delta$ since $k_K(s_0+it, w+it)\to 0$ as $t\to+\infty$).

Another immediate consequence of Theorem \ref{Thm:main} is the following: if $\Delta\subsetneq \C$ is a simply connected domain such that $(\Ha+a)\subset \Delta\subset \Ha$ where $\Ha:=\{z\in \C: \Re z>0\}$ and $a>0$, then $f^{-1}(t)$ converges orthogonally to a point $\tau\in \partial \D$ as $t\to+\infty$.

As we mentioned before, Theorem \ref{Thm:main} has direct applications to the study of the so-called ``slope problem'' for continuous semigroups of holomorphic self-maps of the unit disc (or more generally, for discrete iteration of holomorphic self-maps of the unit disc). We discuss this in Section \ref{semigroups}.

The proof of Theorem \ref{Thm:main} is rather involved. Using invariance of hyperbolic objects by Riemann mappings, one can essentially reduce to the case  $(\Ha+1)\subset \Delta\subset \Ha$. In this case, the curve $\Gamma:(0,+\infty)\ni t\mapsto t\in \C$, for $t$ sufficiently large, is a uniform quasi-geodesic in $\Delta$ in the sense of Gromov. This implies easily that $f^{-1}(\Gamma)$ converges non-tangentially to a point $\sigma\in \partial \D$. However, in order to show that the convergence is orthogonal, one has to make careful estimates of the hyperbolic distance, proving that the geodesic of $\Delta$ which ``shadows'' $\Gamma$ becomes closer and closer in the hyperbolic distance of $\Delta$ to $\Gamma$.  This is the content of Section \ref{theproof}.

\section{Geodesics and quasi-geodesics in simply connected domains}\label{geodesics}

Let $\Delta\subsetneq \C$ be a simply connected domain. We denote by $\kappa_\Delta$ the infinitesimal metric in $\Delta$, that is,  for $z\in \Delta$, $v\in \C$, we let
\[
\kappa_\Delta(z;v):=\frac{|v|}{f'(0)},
\]
where $f:\D\to \Delta$ is the Riemann map such that $f(0)=z$, $f'(0)>0$. The hyperbolic distance $k_\Delta$ in $\Delta$ is defined for $z, w\in \Delta$ as
\[
k_\Delta(z,w):=\inf \int_0^1 \kappa_\Delta(\gamma(t);\gamma'(t))dt,
\]
where the infimum is taken over all piecewise smooth curves $\gamma:[0,1]\to \Delta$ such that $\gamma(0)=z, \gamma(1)=w$.

Let  $-\infty<a<b<+\infty$ and let $\gamma:[a,b]\to \Delta$ be a piecewise $C^1$-smooth curve. For $a\leq s\leq t\leq b$, we define the {\sl hyperbolic length of $\gamma$ in $\Delta$} between $s$ and $t$ as
\[
\ell_\Delta(\gamma;[s,t]):=\int_s^t \kappa_\Delta(\gamma(u);\gamma'(u))du.
\]
In case the length is computed in all the interval $[a,b]$ of definition of the curve, we will simply write
\[
\ell_\Delta(\gamma):=\ell_\Delta(\gamma;[a,b]).
\]
In particular recall that, if $\Ha:=\{z\in \C: \Re z>0\}$, then for every $z, w\in \Ha$,
\begin{equation}\label{Eq:dist-hyp-H}
k_\Ha(z,w):=\frac{1}{2}\log \frac{1+\left|\frac{z-w}{z+\overline{w}} \right|}{1-\left|\frac{z-w}{z+\overline{w}} \right|}.
\end{equation}

\begin{definition}
Let $\Delta\subsetneq \C$ be a simply connected domain.  A    $C^1$-smooth curve $\gamma:(a,b)\to \Delta$, $-\infty\leq a<b\leq +\infty$, such that $\gamma'(t)\neq 0$ for all $t\in (a,b)$ is called a {\sl geodesic} of $\Delta$  if for every $a< s\leq t< b$,
\[
\ell_\Delta(\gamma;[s,t])=k_\Delta(\gamma(s), \gamma(t)).
\]
Moreover, if $z,w\in \Delta$ and there exist $a<s<t<b$ such that $\gamma(s)=z$ and $\gamma(t)=w$, we say that $\gamma|_{[s,t]}$ is a geodesic which joins $z$ and $w$.

With a slight abuse of notation, we also call geodesic the image of $\gamma$ in $\Delta$.
\end{definition}

Using Riemann maps and the invariance of hyperbolic metric and distance for biholomorphisms, we have the following result (see, {\sl e.g.}, \cite{CL} for the definition and properties of the Carath\'eodory's prime ends topology):

\begin{proposition}\label{Prop:geodesic-in-simply}
Let $\Delta\subsetneq \C$ be a simply connected domain. Let $-\infty\leq a<b\leq +\infty$.
\begin{enumerate}
\item If $\eta:(a,b) \to \Delta$ is a geodesic,  then
\[
\eta(a):=\lim_{t\to a^+}\eta(t), \quad \eta(b):=\lim_{t\to b^-}\eta(t)
\]
 exist as limits in the Carath\'eodory topology of $\Delta$. Moreover, if $\eta(a), \eta(b)\in \Delta$ then
 \[
k_\Delta(\eta(a),\eta(b))=\lim_{\epsilon\to 0^+}\ell_{\Delta}(\eta;[a+\epsilon,b-\epsilon]).
 \]
\item If $\eta:(a,b) \to \Delta$ is a geodesic such that $\eta(a), \eta(b)\in \partial_C \Delta$, then $\eta(a)\neq \eta(b)$.
\item For any $z,w\in \widehat{\Delta}$, $z\neq w$, there exists a real analytic geodesic $\gamma:(a,b)\to \Delta$ such that $\gamma(a)=z$ and $\gamma(b)=w$. Moreover, such a geodesic is essentially unique, namely, if $\eta:(\tilde a, \tilde b)\to \Delta$ is another geodesic joining $z$ and $w$, then $\gamma([a,b])=\eta([\tilde a,\tilde b])$ in $\widehat{\Delta}$.
\item If  $\gamma:(a,b)\to \Delta$ is a geodesic such that either $\gamma(a)\in \Delta$ or $\gamma(b)\in \Delta$ (or both), then there exists a geodesic $\eta:(\tilde a,\tilde b)\to \Delta$ such that $\eta(\tilde a), \eta(\tilde b)\in \partial_C \Delta$ and such that $\gamma([a,b])\subset \eta([\tilde a, \tilde b])$ in $\widehat{\Delta}$.
\item If  $\gamma:(a,b)\to \Delta$ is a geodesic such that  $\gamma(a)\in \partial_C\Delta$ then the cluster set $\Gamma(\gamma,a)=\Pi(\gamma(a))$, the principal part of the prime end $\gamma(a)$ (and similarly for $b$ in case $\gamma(b)\in  \partial_C\Delta$).
\end{enumerate}
\end{proposition}

Given a simply connected domain, it is in general a hard task to find geodesics. The aim of this section is indeed to recall a powerful method due to Gromov to localize geodesics via simpler curves which are called quasi-geodesics.

\begin{definition}
Let $\Delta\subsetneq \C$ be a simply connected domain. Let $A\geq 1$ and $B\geq 0$. A piecewise $C^1$-smooth curve $\gamma:[a,b]\to \Delta$, $-\infty<a<b<+\infty$, is a {\sl $(A,B)$-quasi-geodesic} if  for every $a\leq s\leq t\leq b$,
\[
\ell_\Delta(\gamma; [s,t])\leq A k_\Delta(\gamma(s),\gamma(t))+B.
\]
\end{definition}

The importance of quasi-geodesics is contained in the following result (see, {\sl e.g.} \cite{Ghys}):

\begin{theorem}[Gromov's shadowing lemma]\label{Gromov}
For every $A\geq 1$ and $B\geq 0$ there exists $\delta=\delta(A,B)>0$ with the following property. Let  $\Delta\subsetneq \C$ be any simply connected domain. If $\gamma:[a,b]\to \Delta$ is a $(A,B)$-quasi-geodesic, then there exists a geodesic $\tilde\gamma:[\tilde a, \tilde b]\to \Delta$ such that $\tilde\gamma(\tilde a)=\gamma(a)$, $\tilde\gamma(\tilde b)=\gamma(b)$ and for every $u\in [a,b]$ and $v\in [\tilde a, \tilde b]$,
\[
k_\Delta(\gamma(u), \tilde\gamma([\tilde a, \tilde b]))<\delta, \quad k_\Delta(\tilde\gamma(v),\gamma([ a,  b]))< \delta.
\]
\end{theorem}

A consequence of Gromov's shadowing lemma we will take advantage of is the following result, whose proof is based on standard arguments of normality:

\begin{corollary}\label{Cor:shadow}
Let  $\Delta\subsetneq \C$ be a simply connected domain. Let $\gamma:[0,+\infty)\to \Delta$ be a piecewise $C^1$-smooth curve such that $\lim_{t\to +\infty}k_\Delta(\gamma(0), \gamma(t))=+\infty$ and there exist $A\geq 1$, $B\geq 0$, such that for every fixed $T>0$ the curve $[0,T]\ni t\mapsto \gamma(t)$ is a $(A,B)$-quasi-geodesic. Then there exists a prime end $\underline{x}\in \partial_C\Delta$ such that $\gamma(t)\to \underline{x}$ in the Carath\'eodory topology of $\Delta$ as $t\to +\infty$. Moreover, there exists $\epsilon>0$ such that, if $\eta:[0,+\infty)\to \Delta$ is the geodesic of $\Delta$ parameterized by arc length such that $\eta(0)=\gamma(0)$ and $\lim_{t\to+\infty}\eta(t)=\underline{x}$ in the Carath\'eodory topology of $\Delta$, then, for every $t\in [0,+\infty)$,
\begin{equation*}
k_\Delta(\gamma(t), \eta([0,+\infty))<\epsilon, \quad k_\Delta(\eta(t), \gamma([0,+\infty))<\epsilon.
\end{equation*}
\end{corollary}

\section{Horocycles in simply connected domains}\label{horocycles}

In this section we define horocycles (also called ``horospheres'') in simply connected domains, using in this context an abstract approach introduced in \cite{BG}, and inspired on the general definition of horospheres in several complex variables introduced by M. Abate \cite{ABZ, Abate}.

Recall (see {\sl e.g.}, \cite[Prop. 1.2.2]{Abate}) that for every $\sigma\in \partial \D$,
\begin{equation}\label{Eq:yang}
\lim_{w\to \sigma}[k_\D(z,w)-k_\D(0,w)]=\frac12\log\frac{|\sigma-z|^2}{1-|z|^2},
\end{equation}
so that, given $R>0$,
\begin{equation}\label{Eq:yang2}
E(\sigma, R):=\{z\in \D: \frac{|\sigma-z|^2}{1-|z|^2}<R\}=\{z\in \D: \lim_{w\to \sigma}[k_\D(z,w)-k_\D(0,w)]<\frac{1}{2}\log R\}
\end{equation}
is a horocycle of $\D$ of center $\sigma$ and radius $R$.

Let $\Delta\subsetneq \C$ be a simply connected domain,  $z_0\in \Delta$ and let $f:\D\to \Delta$ be a Riemann map such that $f(0)=z_0$. Let $\underline{y}\in \partial_C \Delta$ be a prime end of $\Delta$. There exists exactly one $\sigma\in \partial \D$ which corresponds to a prime end $\underline{x}_\sigma\in \partial_C\D$ such that $\hat{f}(\underline{x}_\sigma)=\underline{y}$. Moreover, a sequence $\{w_n\}\subset \Delta$ converges to $\underline{y}$ in the Carath\'eodory topology of $\Delta$ if and only if $\{f^{-1}(w_n)\}$ converges to $\sigma$ (in the Euclidean topology). Therefore, if $\{z_n\}$ and $\{w_n\}$ are two sequences in $\Delta$ which converge to $\underline{y}$ in the Carath\'eodory topology of $\Delta$, taking into account that $f$ is an isometry for the hyperbolic distance, by \eqref{Eq:yang}, we have for every $z\in \Delta$:
\begin{equation*}
\begin{split}
\lim_{n\to \infty}[k_\Delta(z,w_n)-k_\Delta(z_0,w_n)]&=\lim_{n\to \infty}[k_\D(f^{-1}(z),f^{-1}(w_n))-k_\D(0,f^{-1}(w_n))]\\&=\lim_{n\to \infty}[k_\D(f^{-1}(z),f^{-1}(z_n))-k_\D(0,f^{-1}(z_n))]\\&=\lim_{n\to \infty}[k_\Delta(z,z_n)-k_\Delta(z_0, z_n)].
\end{split}
\end{equation*}
Moreover, by the same equation \eqref{Eq:yang}, $$\lim_{n\to \infty}[k_\D(f^{-1}(z),f^{-1}(w_n))-k_\D(0,f^{-1}(w_n))]\in (-\infty,+\infty),$$ hence  $\lim_{n\to \infty}[k_\Delta(z,w_n)-k_\Delta(z_0,w_n)]\in (-\infty,+\infty)$.

By the previous remark,  the following definition is well posed (that is, it is independent of the sequence $\{w_n\}$ chosen):

\begin{definition}
Let $\Delta\subsetneq \C$ be a simply connected domain and $z_0\in \Delta$. Let $\underline{y}\in \partial_C \Delta$ be a prime end of $\Delta$. Let $R>0$. The horocycle $E_{z_0}^\Delta(\underline{y}, R)$ of {\sl center $\underline{y}$}, {\sl base point $z_0$} and {\sl hyperbolic (radius) $R>0$} is
\[
E_{z_0}^\Delta(\underline{y}, R):=\{z\in \Delta: \lim_{n\to \infty}[k_\Delta(z,w_n)-k_\Delta(z_0,w_n)]<\frac{1}{2}\log R\},
\]
where $\{w_n\}\subset \Delta$ is any sequence which converges to $\underline{y}$ in the Carath\'eodory topology of~$\Delta$.
\end{definition}
Note that, by \eqref{Eq:yang}, $E(\sigma, R)=E^\D_0(\underline{x}_\sigma, R)$ for every $\sigma\in \partial \D$ and $R>0$.

The base point $z_0$ in the definition of horocycles is essentially irrelevant. Indeed, as proven by a direct computation, for every $R>0$, $z_0, z_1\in \Delta$ and $\underline{y}\in \partial_C \Delta$
\[
E_{z_0}^\Delta(\underline{y}, R)=E_{z_1}^\Delta(\underline{y}, AR),
\]
where $A\in (0,+\infty)$ and, in fact,
\[
-k_{\Delta}(z_0,z_1) \leq \frac{1}{2}\log A:=\lim_{n\to \infty}[k_\Delta(z_0,w_n)-k_\Delta(z_1, w_n)] \leq k_{\Delta}(z_0,z_1).
\]
If $\Delta, \widetilde\Delta\subsetneq \C$ are simply connected domain and $f:\Delta\to \widetilde\Delta$ is a biholomorphism, then we denote by $\hat{f}$ the homeomorphism induced by $f$ in the Carath\'eodory topology. Since biholomorphisms are isometries for the hyperbolic distance, we  immediately get:

\begin{proposition}\label{Prop:invariance}
Let $\Delta, \widetilde\Delta\subsetneq \C$ be two simply connected domains. Let $f:\Delta\to \widetilde\Delta$ be a biholomorphism, $\underline{y}\in \partial_C\Delta$ and $z_0\in \Delta$. Then, for every $R>0$
\[
f(E_{z_0}^\Delta(\underline{y}, R))=E_{f(z_0)}^{\widetilde \Delta}(\hat{f}(\underline{y}), R).
\]
\end{proposition}

Now we need a ``localization lemma'', comparing hyperbolic distance in horospheres with hyperbolic distance in the domain. We start with a simple lemma whose proof is a direct computation based on \eqref{Eq:dist-hyp-H}:

\begin{lemma}\label{Lem:hyper-semipiano}
Let $\beta\in (-\frac{\pi}{2},\frac{\pi}{2})$.
\begin{enumerate}
\item Let  $0<\rho_0<\rho_1$ and let $\Gamma:=\{\rho e^{i\beta}: \rho_0\leq \rho\leq \rho_1\}$. Then, $\displaystyle{\ell_{\Ha}(\Gamma)=\frac{1}{2\cos \beta}\log\frac{\rho_1}{\rho_0}}$.
\item Let $\rho_0, \rho_1>0$. Then, $\displaystyle{k_{\Ha}(\rho_0,\rho_1e^{i\beta})-k_{\Ha}(\rho_0,\rho_1)\geq \frac{1}{2}\log\frac{1}{\cos \beta}.}$
\item Let $\rho_0>0$ and $\alpha\in (-\frac{\pi}{2},\frac{\pi}{2})$. Then, $(0,+\infty)\ni \rho\mapsto k_{\Ha}(\rho e^{i\alpha},\rho_0e^{i\beta})$ has a minimum at $\rho=\rho_0$, it is  increasing for $\rho>\rho_0$ and  decreasing for $\rho<\rho_0$.
\item Let $\theta_0, \theta_1\in (-\frac{\pi}{2},\frac{\pi}{2})$ and $\rho>0$. Then $k_{\Ha}(\rho e^{i\theta_0},\rho e^{i\theta_1})=k_{\Ha}(e^{i\theta_0},e^{i\theta_1})$. Moreover, $k_\Ha(1,e^{i\theta})=k_\Ha(1,e^{-i\theta})$ for all $\theta\in [0,\pi/2)$ and $[0, \pi/2)\ni \theta\mapsto k_\Ha(1,e^{i\theta})$ is strictly increasing.
\item Let $\beta_0,\beta_1\in (-\frac{\pi}{2},\frac{\pi}{2})$ and $0<\rho_0<\rho_1$. Then $\displaystyle{
k_{\Ha}(\rho_0e^{i\beta_0},\rho_1e^{i\beta_1})\geq k_{\Ha}(\rho_0,\rho_1)}$.
\end{enumerate}
\end{lemma}

Let $\infty\in\partial_C\Ha$ denote the prime end corresponding to $1 \in \partial \mathbb D$, via the Cayley transform $C:\D\to \Ha$, $C(z)=\frac{1+z}{1-z}$. For $R>0$ a direct computation shows
\[
E_1^\Ha(\infty, R)=\{w\in \C: \Re w>R\}.
\]
In particular, if $t_0>0$ and $\gamma:[1,+\infty)\to \Ha$ is the geodesic defined by $\gamma(t)=t$, it follows that $\gamma(t)\to \infty$ in the Carath\'eodory topology of $\Ha$ and $\gamma(t)$ is eventually contained in $E_1^\Ha(\infty, R)$ for every $R>0$.

Also, if $\gamma:[0,+\infty)$ is a geodesic in a simply connected domain $\Delta\subsetneq \C$ such that $\lim_{t\to +\infty}k_\Delta(\gamma(0), \gamma(t))=+\infty$ and $R>0$ we define
\[
S_\Delta(\gamma, R):=\{z\in \Delta: k_\Delta(z,\gamma)<R\},
\]
a {\sl hyperbolic sector around $\gamma$}.

\begin{lemma}\label{Lem:contril-hyp-horo}
Let $U\subsetneq \C$ be a simply connected domain, $z_0\in U$, $\underline{y}\in \partial_CU$ and  $\gamma:[0,+\infty)\to U$  a geodesic which converges to $\underline{y}$ in the Carath\'eodory topology of $U$. Let $\{z_n\}$ be a sequence with no accumulation points in $U$. If $\lim_{n\to \infty}k_U(z_n, \gamma([0,+\infty)))=0$ then $\{z_n\}$ is eventually contained in $E^U_{z_0}(\underline{y}, R)$ and $\lim_{n\to \infty}k_{E^U_{z_0}(\underline{y}, R)}(z_n, \gamma([0,+\infty)))=0$, for all $R>0$.
\end{lemma}
\begin{proof}
Using Proposition \ref{Prop:invariance} and the fact that horospheres do not depend essentially on the base point, we can assume $U=\Ha$, $z_0=1$, $\underline{y}=\infty\in\partial_C\Ha$ and $\gamma(t)=t$, $t\geq 1$.

Write $z_n=\rho_ne^{i\theta_n}$, with $\rho_n>0$ and $\theta_n\in (-\pi/2,\pi/2)$. Hence, $k_\Ha(\rho_n e^{i\theta_n}, [1,+\infty))=k_\Ha(\rho_n e^{i\theta_n}, \rho_n)$ by  Lemma \ref{Lem:hyper-semipiano}(3). Moreover, by assumption, $\lim_{n\to \infty}k_\Ha(\rho_n e^{i\theta_n}, \rho_n)=0$. Hence,  by Lemma \ref{Lem:hyper-semipiano}(4), $\{z_n\}$ is eventually contained in $V(r, M):=\{\rho e^{i\theta}: \rho>M, |\theta|<r\}$ for all $r>0$, $M>0$.

Note that, given $R>0$, then for $M$ sufficiently large, $V(r,M)\subset E_1^\Ha(\infty, R)$---and, in particular, $\{z_n\}\subset E_1^\Ha(\infty, R)$ eventually.

Let $R>0$. By \cite[Lemma 4.4]{BCDG} and the previous consideration, it follows that for all $N>0$, there exists $t_0\geq 1$ such that $S_\Ha(\gamma|_{[t_0, \infty)}, N)\subset E_1^\Ha(\infty, R)$ and  $\{z_n\}\subset S_\Ha(\gamma|_{[t_0, \infty)}, N)$ eventually.
Hence, by \cite[Lemma 4.6]{BCDG}, once fixed $t_0$ and $N$, and setting $S:=S_\Ha(\gamma|_{[t_0, \infty)}, N)$ for short, there exists a constant $C>0$ such that for all $n$ such that $z_n\in S$,
\[
k_S(\rho_n, z_n)\leq C k_\Ha(\rho_n, z_n).
\]
Therefore, $\lim_{n\to \infty}k_S(\rho_n, z_n)=0$. Now, given $R>0$,  we can choose $N>0$ and $t_0\geq 1$ such that $S:=S_\Ha(\gamma|_{[t_0, \infty)}, N)\subset E_1^\Ha(\infty, R)$. Hence, for $n$ sufficiently large,
\[
k_{E_1^\Ha(\infty, R)}(\rho_n, z_n)\leq k_S(\rho_n, z_n)\to 0.
\]
This proves Lemma~\ref{Lem:contril-hyp-horo}.
\end{proof}

\section{Proof of Theorem \ref{Thm:main}}\label{theproof}

For $\beta\in (0,\pi)$, we denote
\[
V(\beta):=\{\rho e^{i\theta}: \rho>0, |\theta|<\beta\}.
\]
By Lemma \ref{Lem:hyper-semipiano}, it follows immediately that $V(\beta)$ is a hyperbolic sector around the geodesic $(0,+\infty)$ of $\Ha$.

The main result we need is the following:

\begin{proposition}\label{orthog-in-semipl}
Let $\Delta\subsetneq \C$ be a simply connected domain and let $f:\D \to \Delta$ be a Riemann map. Suppose that $\Ha+a\subset \Delta\subset \Ha$ for some $a>0$. Then there exists $\xi\in \partial \D$ such that $f^{-1}(t)$ converges orthogonally to $\xi$ as $t\to +\infty$.
\end{proposition}

\begin{proof}
We can assume without loss of generality that $a=1$ and let $U:=\Ha+1$. Note that, for every $z,w\in U$,
\begin{equation}\label{Eq:basic-move-HU}
k_U(z,w)=k_\Ha(z-1,w-1).
\end{equation}
We divide the proof in several steps:

\smallskip

{\sl Step 1.} Let $\beta\in (0,\pi/4)$. Then there exists a constant $K(\beta)>0$ such that for every $\theta_0, \theta_1\in [-\beta, \beta]$ and for every $\rho\geq 2$, $k_U(\rho e^{i\theta_0}, \rho e^{i\theta_1})<K(\beta)$. Moreover, $\lim_{\beta\to 0}K(\beta)=0$.

\smallskip

Fix $\beta\in (0,\frac{\pi}{4})$ and $\rho\geq 2$.
Let $\tilde{\beta}=\tilde{\beta}(\rho)\in (0, \frac{\pi}{2})$ be such that $e^{i\tilde{\beta}}=\frac{\rho e^{i\beta}-1}{|\rho e^{i\beta}-1|}$.
Hence,
\[
\sin \tilde \beta=\frac{\sin \beta}{|e^{i\beta}-1/\rho|}\leq\frac{2\sin \beta}{|2e^{i\beta}-1|},
\]
which shows that $\lim_{\beta\to 0}\sup_{\rho\geq 2}\tilde \beta(\rho)=0$.

Let $A_1:=\{|e^{i\beta}\rho-1|e^{i\theta}+1: |\theta|\leq \tilde\beta\}$. Note that $A_1$ is the  arc of the circle  with center $1$, radius $|e^{i\beta}\rho-1|$  and end points $p_0:=e^{i\beta}\rho=|e^{i\beta}\rho-1|e^{i\tilde{\beta}}+1$ and $p_1:=e^{-i\beta}\rho=|e^{i\beta}\rho-1|e^{-i\tilde{\beta}}+1$.

Let $A_2:=\{(\rho-1) e^{i\theta}+1 : |\theta|\leq \tilde	\beta\}$. Note that $A_2$ is the arc of the circle with center $1$ and radius $\rho-1$ with end points $q_0:=(\rho-1) e^{i\tilde\beta}+1$ and $q_1:=(\rho-1) e^{-i\tilde\beta}+1$.

Note that by construction, $A_1, A_2$ are arcs of circles that intersect $\partial U$ orthogonally, hence they are geodesics for the hyperbolic distance $k_{U}$.

Let $B_1:=\{re^{i\tilde\beta}+1: \rho-1\leq r\leq |e^{i\beta}\rho-1| \}$ and  let $B_2:=\{re^{-i\tilde\beta}+1: \rho-1\leq r\leq |e^{i\beta}\rho-1| \}$.

By construction,  $A_1\cup B_1\cup A_2\cup B_2$ is a Jordan curve which bounds a simply connected domain $Q\subset \C$. Moreover, by simple geometric considerations, the curve $\{\rho e^{i\theta}: |\theta|\leq \beta\}$ is contained in $Q$. Hence,
\[
k_{U}(\rho e^{i\theta_0},\rho e^{i\theta_1})\leq \hbox{diam}_{U}(Q),
\]
where $\hbox{diam}_{U}(Q):=\sup_{z,w\in Q}k_{U}(z,w)$ is the hyperbolic diameter of $Q$. Clearly,
\[
\hbox{diam}_{U}(Q)\leq \ell_{U}(A_1)+\ell_{U}(A_2)+\ell_{U}(B_1)+\ell_{U}(B_2).
\]

Now, since $A_1, A_2$ are geodesics for $U$, it follows that $\ell_{U}(A_1)=k_{U}(p_0,p_1)$ and $\ell_{U}(A_2)=k_{U}(q_0,q_1)$.
Hence, by \eqref{Eq:basic-move-HU}
\[
\ell_{U}(A_1)=k_{U}(p_0,p_1)=k_{\Ha}(p_0-1,p_1-1)=k_{\Ha}(|e^{i\beta}\rho-1|e^{i\tilde{\beta}}, |e^{i\beta}\rho-1|e^{-i\tilde{\beta}}),
\]
and, by   Lemma \ref{Lem:hyper-semipiano}(4), $k_{\Ha}(|e^{i\beta}\rho-1|e^{i\tilde{\beta}}, |e^{i\beta}\rho-1|e^{-i\tilde{\beta}})$ depends on $\tilde\beta$ and   goes to 0 as $\beta$ goes to $0$. Similarly, $\ell_{U}(A_2)$  goes to 0 as $\beta$ goes to $0$.

On the other hand, by \eqref{Eq:basic-move-HU} and Lemma \ref{Lem:hyper-semipiano}(1),
\begin{equation*}
\begin{split}
\ell_{U}(B_1)=\ell_{\Ha}(B_1-1)=\ell_{\Ha}(\{re^{i\tilde\beta}: \rho-1\leq r\leq |e^{i\beta}\rho-1| \})=\frac{1}{2\cos \tilde\beta}\log\frac{|e^{i\beta}\rho-1|}{\rho-1}.
\end{split}
\end{equation*}
Since $\sin(\tilde \beta)\leq \frac{2\sin(\pi/4)}{|2e^{i\pi/4}-1|} = \sqrt{\frac{2}{5-2\sqrt 2}}$, we have that $\cos(\tilde \beta)\geq \sqrt{\frac{3-2\sqrt 2}{5-2\sqrt 2}}$ and
\[
\ell_{U}(B_1)\leq \sqrt{\frac{5-2\sqrt 2}{3-2\sqrt 2}}\log\frac{|e^{i\beta}\rho-1|}{\rho-1}\leq \sqrt{\frac{5-2\sqrt 2}{3-2\sqrt 2}}\log|2e^{i\beta}-1|,
\]
for every $\rho\geq 2$,
which shows that $\ell_{U}(B_1)$ goes to 0 as $\beta$ goes to 0.  A similar argument shows that also $\ell_{U}(B_2)$ goes to 0 as $\beta$ goes to 0 and Step 1 follows.

\smallskip

{\sl Step 2.} Let $\beta\in (0,\pi/2)$. Let $\alpha_\beta:=(1-\cos^2\beta)^{-1}$. Then for every $x_1>x_0\geq \alpha_\beta$ the geodesic in $\Delta$ joining $x_0$ and $x_1$ is contained in $V(\beta)$.

\smallskip

Fix $x_0\geq  \alpha_\beta$ and $x_1>x_0$. Let $\sigma:[0,1]\to \Delta$ be the geodesic for $\Delta$ such that $\sigma(0)=x_0$ and $\sigma(1)=x_1$. Assume that $\sigma([0,1])$ is not contained in $V(\beta)$. Hence, there exist $0< t_1\leq t_2< 1$ such that $\sigma(t_j)\in \partial V(\beta)$, $j=1,2$ and $\{\sigma(t) : t_1\leq t\leq t_2\}\cap V(\beta)=\emptyset$. Since $V(\beta)$ disconnects $\Ha$ in two connected components, we can assume without loss of generality that $\sigma(t_1)=y_1e^{i\beta}$ and $\sigma(t_2)=y_2e^{i\beta}$ for some $y_1, y_2>0$ (possibly $y_1=y_2$).


Denote $R:=\{r: r>0\}$. Let $\gamma_1$ be the segment in $R$ joining $x_0$ and $y_1$, namely, if $y_1\geq x_0$, let $\gamma_1:=\{r: x_0\leq r\leq y_1\}$, while, if $y_1<x_0$, let $\gamma_1:=\{r: y_1\leq r\leq x_0\}$. Let $\gamma_2=\{y_1e^{i\theta}: \theta\in [0,\beta]\}$. Let $\gamma_3$ be the segment on $\partial V(\beta)$ joining $\sigma(t_1)=y_1e^{i\beta}$ with $\sigma(t_2)=y_2e^{i\beta}$, {\sl i.e.}, if for instance $y_1\leq y_2$, $\gamma_3:=\{re^{i\beta}: y_1\leq r\leq y_2\}$. Then, let $\gamma_4:=\{y_2e^{i\theta}: \theta\in [0,\beta]\}$. Finally, let $\gamma_5$ be the segment on $R$ joining $y_2$ with $x_1$.

Let $\Gamma=\gamma_1\cup\gamma_2\cup\gamma_3\cup\gamma_4\cup\gamma_5$. Hence, $\Gamma$ is a piecewise smooth curve in $\Ha$ which joins $x_0$ and $x_1$.

Now, since $\sigma$ is a geodesic in $\Delta$,  $\Delta\subset \Ha$ and by Lemma \ref{Lem:hyper-semipiano}(5),
\begin{equation*}
\begin{split}
\ell_\Delta(\sigma)&=k_\Delta(x_0,y_1e^{i\beta})+k_\Delta(y_1e^{i\beta},y_2e^{i\beta})+k_\Delta(y_2e^{i\beta},x_1)\\
&\geq k_{\Ha}(x_0,y_1e^{i\beta})+k_{\Ha}(y_1e^{i\beta},y_2e^{i\beta})+k_{\Ha}(y_2e^{i\beta},x_1)\\
&\geq k_{\Ha}(x_0,y_1e^{i\beta})+k_{\Ha}(y_1,y_2)+k_{\Ha}(y_2e^{i\beta},x_1)\\
&= k_{\Ha}(x_0,y_1)+k_{\Ha}(y_1,y_2)+k_{\Ha}(y_2,x_1)\\&\qquad +[k_{\Ha}(x_0,y_1e^{i\beta})-k_{\Ha}(x_0,y_1)]+[k_{\Ha}(y_2e^{i\beta},x_1)-k_{\Ha}(y_2,x_1)]\\&\geq k_{\Ha}(x_0,x_1)+[k_{\Ha}(x_0,y_1e^{i\beta})-k_{\Ha}(x_0,y_1)]+[k_{\Ha}(y_2e^{i\beta},x_1)-k_{\Ha}(y_2,x_1)]\\& \geq k_{\Ha}(x_0,x_1)+\log\frac{1}{\cos \beta},
\end{split}
\end{equation*}
where the penultimate inequality follows from the triangle inequality, and the last inequality follows from  Lemma \ref{Lem:hyper-semipiano}(2). Moreover,  a direct computation shows that
\begin{equation*}
k_{\Ha}(x_0,x_1)=k_{U}(x_0,x_1)-\frac{1}{2}\log\frac{1-\frac{1}{x_1}}{1-\frac{1}{x_0}}.
\end{equation*}
Therefore, taking into account that $U\subset \Delta$, from the previous inequality we have
\begin{equation*}
\begin{split}
k_\Delta(x_0,x_1)&=\ell_\Delta(\sigma)\geq k_{U}(x_0,x_1)-\frac{1}{2}\log\frac{1-\frac{1}{x_1}}{1-\frac{1}{x_0}}+\log\frac{1}{\cos \beta}\\
&\geq k_\Delta(x_0,x_1)-\frac{1}{2}\log\frac{1-\frac{1}{x_1}}{1-\frac{1}{x_0}}+\log\frac{1}{\cos \beta},
\end{split}
\end{equation*}
which forces
\[
\log\frac{1}{\cos^2 \beta}\leq \log\frac{1-\frac{1}{x_1}}{1-\frac{1}{x_0}}.
\]
However, if $x_0\geq \alpha_\beta$,
\[
\log\frac{1}{\cos^2 \beta}\leq \log\frac{1-\frac{1}{x_1}}{1-\frac{1}{x_0}}\leq \log\frac{1-\frac{1}{x_1}}{\cos^2\beta},
\]
getting a contradiction, and Step 2 follows.

\smallskip

{\sl Step 3.} Let $\beta\in (0,\pi/4)$ and let $2\leq x_0<x_1$. Let $\sigma:[0,1]\to \Delta$ be a geodesic for $\Delta$ such that $\sigma(0)=x_0$ and $\sigma(1)=x_1$. Let $K(\beta)$ be the constant defined in Step 1 and let $c\in (0, x_0e^{-K(\beta)})$. Suppose $\sigma([0,1])\subset V(\beta)$. Then $|\sigma(t)|> c$ for all $t\in [0,1]$.

\smallskip

Assume by contradiction that there exists $t_1\in (0,1)$ such that $|\sigma(t_1)|=  c$. Then $\sigma(t_1)=c e^{i\theta_1}$  for some $\theta_1\in (-\beta,\beta)$. Moreover,  by continuity of $\sigma$, there exist $\tilde t_1\in [0,t_1)$ and $\tilde t_2\in (t_1, 1)$ such that $\sigma(\tilde t_1)=x_0e^{i\tilde\theta_1}$, $\sigma(\tilde t_2)=x_0e^{i\tilde\theta_2}$ for some $\tilde\theta_1, \tilde\theta_2\in (-\beta,\beta)$ and $|\sigma(t)|\leq x_0$ for all $t\in [\tilde t_1,\tilde t_2]$.

Now, since $\sigma$ is a geodesic in $\Delta$, and $\Delta\subset \Ha$,
\begin{equation*}
\begin{split}
k_\Delta(x_0e^{i\tilde\theta_1}, x_0e^{i\tilde\theta_2})&=\ell_\Delta(\sigma;[\tilde t_1, \tilde t_2])= \ell_\Delta(\sigma;[ \tilde t_1, t_1])+\ell_\Delta(\sigma;[ t_1, \tilde t_2])\\&=k_\Delta(x_0e^{i\tilde\theta_1}, ce^{i\theta_1})+k_\Delta(ce^{i\theta_1},x_0e^{i\tilde\theta_2})\\
& \geq k_{\Ha}(x_0, c)+k_{\Ha}(c,x_0)=2k_{\Ha}(x_0, c)=\log\frac{x_0}{c},
\end{split}
\end{equation*}
where the last inequality follows from Lemma \ref{Lem:hyper-semipiano}(5).

On the other hand, since $U\subset \Delta$, and by Step 1,
\[
k_\Delta(x_0e^{i\tilde\theta_1}, x_0e^{i\tilde\theta_2})\leq k_{U}(x_0e^{i\tilde\theta_1}, x_0e^{i\tilde\theta_2})\leq K(\beta).
\]
 Hence, $\log\frac{x_0}{c}\leq K(\beta)$, which contradicts the choice of $c$ and Step 3 follows.

\smallskip

{\sl Step 4.} For every $\delta>0$ there exists $\mu_\delta\geq 2$ such that for every $x_1>x_0\geq \mu_\delta$, if  $\sigma:[0,1]\to \Delta$ is a geodesic of $\Delta$ such that $\sigma(0)=x_0$ and $\sigma(1)=x_1$, then for every $x\in [x_0, x_1]$ there exists $t_x\in [0,1]$ such that $k_\Delta(x, \sigma(t_x))<\delta$.

\smallskip

In order to prove Step 4, we first claim that for every $\nu\in (0,\frac{\pi}{4})$ there exists $\mu_\nu\geq 2$ such that for every $x_1>x_0\geq \mu_\nu$, $\sigma([0,1])\subset V(\nu)+1$.

If the claim is true, since $\sigma$ is continuous, for every $x\in [x_0,x_1]$ there exist $|\theta_x| <\nu$ and $t_x\in [0,1]$ such that
$\sigma(t_x)=(x-1)e^{i\theta_x}+1$.  Hence, by \eqref{Eq:basic-move-HU} and  Lemma \ref{Lem:hyper-semipiano}(4), and recalling that  $U\subset \Delta$,
\begin{equation*}
\begin{split}
k_{\Delta}(\sigma(t_x), x)&\leq k_{U}((x-1)e^{i\theta_x}+1, (x-1)+1)=k_\Ha((x-1)e^{i\theta_x}, x-1)\\&=k_\Ha(e^{i\theta_x},1)< k_\Ha(e^{i\nu}, 1).
\end{split}
\end{equation*}
Since $k_\Ha(e^{i\nu}, 1)\to 0$ as $\nu\to 0$, Step 4 follows.

In order to prove the claim, given $\nu\in (0,\frac{\pi}{4})$, let $\beta\in (0,\nu)$. Hence, there exists $\alpha>1$ such that $V(\beta)\cap\{w\in U: |w|>\alpha\}\subset V(\nu)+1$. Let $\alpha_\beta$ be given by Step~2. Let $\mu_\nu>e^{K(\beta)}\max\{\alpha_\beta,\alpha\}$, where $K(\beta)$ is given by Step~1. Hence, by Step~2, for every $x_1>x_0>\mu_\nu$ the geodesic $\sigma$ for $k_\Delta$ joining $x_0, x_1$ is contained in $V(\beta)$. By Step~3, $\sigma([0,1])\subset \{w: |w|>\alpha\}$, hence, $\sigma$ is contained in $V(\nu)+1$.

\smallskip

Let $\xi\in \partial \D$ be such that $\lim_{t\to+\infty}f^{-1}(t)=\xi$ (see \cite[Page 162]{Shabook83}). Let $\gamma:[0,1)\to \D$ be the geodesic of $\D$ defined by $\gamma(t)=t\xi$.

\smallskip

{\sl Step 5.}  For every  $\epsilon>0$ there exists $t_\epsilon>0$ such that $f^{-1}(t)\in S_\D(\gamma, \epsilon)$ for all $t\geq t_\epsilon$.

\smallskip

Fix $\epsilon>0$. Let $\delta=\frac{\epsilon}{3}$ and let $\mu_\delta\geq 2$ be the point defined in Step 4. Let $\{x_n\}$ be an increasing sequence of positive real numbers converging to $+\infty$. Let $\sigma_n:[0, R_n]\to \Delta$ be the geodesic in $\Delta$ parameterized by arc length such that $\sigma_n(0)=\mu_\delta$ and $\sigma_n(R_n)=x_n$. Using a normality argument, up to extracting a subsequence, we can assume that $\{\sigma_n\}$ converges uniformly on compacta of $[0,+\infty)$ to a geodesic $\sigma:[0,+\infty)\to \Delta$, parameterized by arc length such that $\sigma(0)=\mu_\delta$ and $\lim_{s\to +\infty}\sigma(s)=\underline{y}\in \partial_C\Delta$ in the Carath\'eodory topology of $\Delta$.

In particular, for every fixed $T>0$ there exists $n_{T}\in \N$ such that for every $n\geq n_{T}$ we have $R_n\geq T$  and for every $s\in [0, T]$,
\begin{equation}\label{Eq:conv-u-g-orth}
k_\Delta(\sigma_n(s), \sigma(s))<\delta.
\end{equation}
By Step 4, for every $t\in [\mu_\delta, x_n]$ there exists $s_t^n\in [0, R_n]$ such that $k_\Delta(\sigma_n(s_t^n), t)<\delta$.

We claim that, for every fixed $x_1>\mu_\delta$ there exists $C_{x_1}>0$   such that for all $n\in \N$ and all $t\in [\mu_\delta, x_1]$, we have $s_t^n\leq C_{x_1}$. Indeed, since $[\mu_\delta, x_1]$ is compact in $\Delta$, $C_0:=\max_{x\in [\mu_\delta, x_1]}k_\Delta(x, \mu_\delta)<+\infty$. Hence, recalling that $\sigma_n$ is parameterized by arc length, for all $t\in [\mu_\delta, x_1]$, we have
\begin{equation*}
s_t^n=k_\Delta(\sigma_n(s_t^n), \sigma_n(0))=k_\Delta(\sigma_n(s_t^n), \mu_\delta)\leq k_\Delta(\sigma_n(s_t^n), t)+k_\Delta(t, \mu_\delta)\leq \delta+C_0=:C_{x_1}.
\end{equation*}

Therefore, fix $x_1>\mu_\delta$, and set $T:=C_{x_1}$. By \eqref{Eq:conv-u-g-orth}, for all $t\in [\mu_\delta, x_1]$ we have
\[
k_\Delta(\sigma(s_t^{n_T}), t)\leq k_\Delta(\sigma(s_t^{n_T}), \sigma_{n_T}(s_t^{n_T}))+k_\Delta(\sigma_{n_T}(s_t^{n_T}), t)<2\delta.
\]
By the arbitrariness of $x_1$, this proves that $t\in S_\Delta(\sigma,2\delta)$ for all $t\geq \mu_\delta$.

Since $f$ is an isometry for the hyperbolic distance, $f^{-1}\circ \sigma$ is a geodesic in $\D$ parameterized by arc length and $f^{-1}(t)\in S_\D(f^{-1}\circ \sigma;2\delta)$ for all $t\geq \mu_\delta$.
In particular, for every $t\geq \mu_\delta$ there exists $s_t\in [0,+\infty)$ such that
\begin{equation}\label{Eq:distanza-minima-geo-11}
k_\D(f^{-1}(t), f^{-1}(\sigma(s_t)))<2\delta.
\end{equation}
Note that this implies in particular that $s_t\to+\infty$ as $t\to +\infty$. Hence, $\lim_{t\to+\infty}f^{-1}(\sigma(s_t))=\xi$ and $\lim_{t\to +\infty}f^{-1}(\sigma(t))=\xi$.

Using a Cayley transform from $\D$ to $\Ha$ which maps $\xi$ to $\infty$, it follows that there exists $s_1\geq 0$ such that $f^{-1}(\sigma(s))\in S_\D(\gamma;\delta)$ for all $s\geq s_1$.  In particular, for every $s\geq s_1$, there exists $r_s\in [0,1)$ such that
\begin{equation}\label{Eq:distanza-minima-geo-12}
k_\D(f^{-1}(\sigma(s)), \gamma(r_s))<\delta.
\end{equation}
Let $t_{\epsilon}\geq \mu_\delta$ be such that $s_t\geq s_1$ for all $t\geq t_\epsilon$. Then by \eqref{Eq:distanza-minima-geo-11} and \eqref{Eq:distanza-minima-geo-12}, for all $t\geq t_\epsilon$,
\[
k_\D(\gamma(r_{s_t}), f^{-1}(t))\leq k_\D(f^{-1}(\sigma(s_t)), \gamma(r_{s_t}))+k_\D(f^{-1}(t), f^{-1}(\sigma(s_t)))<3\delta=\epsilon,
\]
and Step 5 follows.

\smallskip

From Step 5 we see that $\limsup_{t\to+\infty}k_\D(f^{-1}(t), [0,1)\xi)=0$, hence, $f^{-1}(t)$ converges to $\xi$ orthogonally.
\end{proof}

Now, we may prove Theorem \ref{Thm:main}:

\begin{proof}[Proof of Theorem \ref{Thm:main}]
If $\{z_n\}\subset \Delta$ is such that $\{f^{-1}(z_n)\}$ converges orthogonally to some $\sigma\in \partial \D$, then the result follows at once taking $U=\Delta$ and any $R>0$. Indeed, by Proposition \ref{Prop:invariance} we can assume $\Delta=U=\D$, and the statement is immediate.

Conversely,  assume that Conditions (1) and (2) of Theorem~\ref{Thm:main} are satisfied. By Proposition \ref{Prop:invariance}, we can assume $U=\Ha$ and $\gamma(t)=t$, for every $t\geq 1$. Since $E^{\mathbb H}_{1}(\infty,R)=\{z \in \mathbb C : \Re \ w > R\}= \mathbb H + R$, then $\mathbb H  + R \subset \Delta \subset \mathbb H$ by Condition (1) and it follows from Proposition \ref{orthog-in-semipl} that there exists $\sigma\in \partial \D$ such that $f^{-1}(\gamma(t))\to \sigma$ orthogonally as $t\to+\infty$.
In particular, this implies that, if $\eta:[0,1)\to \D$ is the geodesic $\eta(t)=t\sigma$, then for every $\epsilon>0$ there exists $t_0\geq 1$ such that for all $t\geq t_0$,
\begin{equation}\label{Eq:gamma-conv-rad-s}
k_\D(f^{-1}(\gamma(t)), \eta([0,1)))<\epsilon.
\end{equation}
On the other hand, by Condition (2) of Theorem~\ref{Thm:main} and by Lemma \ref{Lem:contril-hyp-horo}, for every $\epsilon>0$ there exists $n_0$ such that for all $n\geq n_0$
\begin{equation}\label{Eq:gamma-conv-rad-s2}
k_\D(f^{-1}(z_n), f^{-1}(\gamma([0,+\infty))))=k_\Delta(z_n,\gamma([0,+\infty)))\leq k_{E^{\mathbb H}_{1}(\infty, R)}(z_n, \gamma([0,+\infty)))<\epsilon.
\end{equation}
Note that, since $\{z_n\}$ is compactly divergent in $\Delta$ --- and so is $\{f^{-1}(z_n)\}$ in $\D$ --- this implies in particular that $\{f^{-1}(z_n)\}$ converges to $\sigma$.

By the triangle inequality, \eqref{Eq:gamma-conv-rad-s} and \eqref{Eq:gamma-conv-rad-s2} imply that for every $\epsilon>0$ there exists $n_1$ such that
$k_\D(f^{-1}(z_n), \eta([0,1)))<\epsilon$  for all $n\geq n_1$. Therefore, $\{f^{-1}(z_n)\}$ converges orthogonally to $\sigma$.
\end{proof}

\section{Applications to semigroups}\label{semigroups}

In this section we state some direct but interesting consequences of Theorem \ref{Thm:main} for the dynamics of semigroups in $\D$. Similar results hold for discrete dynamics of holomorphic self-maps of $\D$. We leave details to the interested reader.

A {\sl one-parameter continuous semigroup} of holomorphic self-maps of $\D$---or, for short, semigroup in $\D$---is a continuous homomorphism of the real semigroup $[0,+\infty)$ endowed with the Euclidean topology to the semigroup under composition of holomorphic self-maps of $\D$ endowed with the topology of uniform convergence on compacta. Semigroups in $\D$ have been largerly studied (see, {\sl e.g.}, \cite{Abate,Berkson-Porta, Shb}). It is known that, if $(\phi_t)$ is a semigroup in $\D$, which is not a group of hyperbolic rotations, then there exists $\tau\in \oD$, the {\sl Denjoy-Wolff point} of $(\phi_t)$, such that $\lim_{t\to +\infty}\phi_t(z)=\tau$, and the convergence is uniform on compacta. In case $\tau\in \D$, the semigroup is called elliptic. Non-elliptic semigroups can be divided into three types: hyperbolic, parabolic of positive hyperbolic step and parabolic of zero hyperbolic step. It is known (see \cite{CD, CDP}) that if $(\phi_t)$ is a hyperbolic semigroup then $\{\phi_t(z)\}$ always converges non-tangentially to its Denjoy-Wolff point as $t\to +\infty$ for every $z\in \D$, while, if it is parabolic of positive hyperbolic step then $\{\phi_t(z)\}$ always converges tangentially to its Denjoy-Wolff point as $t\to +\infty$ for every $z\in \D$.

In case of parabolic semigroups of zero hyperbolic step, the behavior of trajectories can be rather wild. All the trajectories have the same slope, that is the cluster set of $\arg(1-\overline{\tau}\phi_t(z))$ as $t\to +\infty$---which is a compact subset of $[-\pi/2,\pi/2]$---does not depend on $z$ (see \cite{CD, CDP}). In most cases this slope is just a point, but in  \cite{Bet, CDG} examples are constructed  such  that the slope is the full interval $[-\pi/2,\pi/2]$ and in \cite{Bet, BCDG} examples are constructed where the slope is a closed subset of $(-\pi/2,\pi/2)$ not reduced to a single point.

Recall that (see, {\sl e.g.}, \cite{Abate, BrAr, Cowen, BCD})  $(\phi_t)$ is a parabolic semigroup in $\D$ of zero hyperbolic step if and only if  there exists a univalent function $h$, the {\sl K\"onigs function} of $(\phi_t)$ such that $h(\D)$ is starlike at infinity, $h(\phi_t(z))=h(z)+it$ for all $t\geq 0$ and $z\in \D$, and for every $w\in \C$ there exists $t_0\geq 0$ such that $w+it_0\in h(\D)$. The triple $(\C, h, z+it)$ is called a {\sl canonical model} for $(\phi_t)$ and it is essentially unique.

Using Theorem \ref{Thm:main} applied to $h(\D)$, we can find geometric conditions that guarantee that the slope reduces to $\{0\}$.

We recall that for $\beta\in (0,\pi)$, $V(\beta):=\{\rho e^{i\theta}: \rho>0, |\theta|<\beta\}$.

\begin{corollary}\label{cor}
Let $(\phi_t)$ be a parabolic semigroup in $\D$ of zero hyperbolic step and with canonical model $(\C, h, z+it)$ and Denjoy-Wolff point $\tau\in \partial \D$. If
\begin{enumerate}
\item either there exists $a>0$ such that $(i\Ha+ia)\subset \Delta \subset i\Ha$,
\item or, there exist $\beta\in (0,\pi)$, $a>0$, such that $(iV(\beta)+ia)\subset \Delta \subset i V(\beta)$,
\item  or, there exist $-\infty<a<b<+\infty$ and $c\in \R$ such that $\partial \Delta$ is contained in the semistrip $\{\zeta\in \C: a<\Re \zeta<b, \Im \zeta<c\}$,
\end{enumerate}
then $\phi_t(z)\to \tau$ orthogonally as $t\to+\infty$, for every $z\in \D$.
\end{corollary}

Statement (3) in Corollary~\ref{cor} was proved by D. Betsakos in \cite{Bet}, as a consequence of Theorem~2 in \cite{Bet}.

\begin{proof}
Statements (1) and (2) follow straightforwardly from Theorem \ref{Thm:main}, considering, respectively, $U=\Ha$ and $U=i V(\beta)$ and, for (2), computing horocycles in $i V(\beta)$ using a Riemann map from $\Ha$ to $i V(\beta)$.

As for (3), since $\Delta$ is starlike at infinity, if $p\in \partial \Delta$, then
\[
\Delta\subset \C\setminus \{\zeta\in \C: \Re \zeta=\Re p, \Im \zeta\leq \Im p\}=:U.
\]
Therefore, using the Koebe map from $\D$ onto $U$ in order to compute horocycles in $U$, one obtains (3) directly from Theorem \ref{Thm:main}.
\end{proof}

\end{document}